\documentclass[notitlepage,11pt,reqno]{amsart}
\usepackage{color}
\usepackage{amssymb,nicefrac,bm,upgreek,mathtools,verbatim,enumerate,ulem}
\usepackage[mathscr]{eucal}
\usepackage{graphicx}
\usepackage{epsf,times}

\definecolor{dmagenta}{rgb}{.4,.1,.5}
\definecolor{007}{rgb}{.0,.0,.7}
\definecolor{dred}{rgb}{.5,.0,.0}
\definecolor{dgreen}{rgb}{.0,.5,.0}
\definecolor{dblue}{rgb}{.0,.0,.5}

\definecolor{violet}{rgb}{.3,.0,.9}
\definecolor{orange}{cmyk}{0,.5,.1,.0}
\definecolor{dcyan}{cmyk}{.5,.0,.0,.0}
\definecolor{dyellow}{cmyk}{.0,.0,.5,.0}
\definecolor{cm}{cmyk}{1,.0,.0,.0}
\numberwithin{equation}{section}
\allowdisplaybreaks

\newtheorem{theorem}{Theorem}[section]
\newtheorem{lemma}{Lemma}[section]

\theoremstyle{definition}

\theoremstyle{remark}

\newcommand{\bvnorm}[1]{[\kern-0.45ex[\kern0.1ex #1 \kern0.1ex]\kern-0.45ex]}

\begin{document}

\title[]
{Existence and uniqueness of solutions to non-Abelian multiple vortex equations on graphs} 

\author[ Y. Hu ]{Yuanyang Hu$^1$ }

\thanks{$^1$ School of Mathematics and Statistics,
	Henan University, Kaifeng, Henan 475004, P. R. China.}
\email{}
\thanks{{\bf Emails:}  {\sf yuanyhu@mail.ustc.edu.cn} (Y. Hu).}

\date{}

\begin{abstract}
Let $G=(V,E)$ be a connected finite graph.
We study a system of non-Abelian multiple vortex equations on $G$. We establish a necessary and sufficient condition for the existence and uniqueness of solutions to the non-Abelian multiple vortex equations.
\end{abstract}

\maketitle

\textit{ \footnotesize Mathematics Subject Classification (2010)}  {\scriptsize 35A01, 35R02}.

\textit{ \footnotesize Key words:} {\scriptsize variational method, vortex, finite graph, equation on graphs}

\section{Introduction}
Vortices play important roles in many areas of theoretical physics including condensed-matter physics, cosmology, superconductivity theory, optics, electroweak theory, and quantum Hall effect. In the past two decades, the topological, non-topological and doubly periodic multivortices to self-dual Chern-Simons model, Chern-Simons Higgs model, the generalized self-dual Chern-Simons model, Abelian Higgs model, the generalized Abelian Higgs model and non-Abelian Chern–
Simons model were established; see, for example, \cite{CI, Ha, LPY, NT, T, TY, Y} and the references therein. Wang and Yang \cite{WY} studied Bogomol’nyi system arising in the abelian Higgs theory defined on a rectangular domain and subject
to a ’t Hooft type periodic boundary condition and established a sufficient and necessary condition for the existence of multivortex solutions of the Bogomol'nyi system. Caffarelli and Yang \cite{CY} established the existence of periodic multivortices in the Chern-Simons Higgs Model. In particular, Lin and Yang \cite{LY} investigated a system of non-Ablian multiple vortex equations governing coupled $SU(N)$ and $U(1)$ gauge and Higgs fields which may be embedded in a supersymmetric field theory framework.

In recent years, equations on graphs have attracted extensive attention; see, for example, \cite{ ALY, Bdd, GJ, GC, Hu, HWY, HLY, LCT, LP, WN} and the references therein. Ge, Hua and Jiang \cite{GHJ} proved that there exists a uniform lower bound for the energy, $\sum\limits_{G} e^{u}$ of any solution $u$ to the equation $\Delta u+e^{u}=0$ on graphs.  Huang, Wang and Yang \cite{HWY} studied the Mean field equation and the relativistic Abelian Chern-Simons equations (involving two  Higgs particles and any two gauge fields) on any finite connected graphs and eatablished some existence results. Huang, Lin and Yau \cite{HLY} proved the existence of solutions to the following mean field equations 
$$
\Delta u+e^{u}=\rho \delta_{0}
$$
and
$$
\Delta u=\lambda e^{u}\left(e^{u}-1\right)+4 \pi \sum_{j=1}^{M} \delta_{p_{j}}
$$
on graphs.
 
 Let $G=(V,E)$ be a connected finite graph, $V$ denote the vetex set and $E$ denote the edge set.
 
 Inspired by the work of Huang-Lin-Yau \cite{HLY}, we investigate a system of non-Abelian multiple vortex equations  
\begin{equation}\label{1}
	\begin{aligned}
		&\Delta u_{1}=-N m_{e}^{2}+m_{e}^{2}\left(\mathrm{e}^{\frac{u_{1}}{N}+\frac{(N-1)}{N} u_{2}}+[N-1] \mathrm{e}^{\frac{u_{1}}{N}-\frac{u_{2}}{N}}\right)+4 \pi \sum_{j=1}^{n} \delta_{p_{j}}(x), \\
		&\Delta u_{2}=m_{g}^{2}\left(\mathrm{e}^{\frac{u_{1}}{N}+\frac{(N-1)}{N} u_{2}}-\mathrm{e}^{\frac{u_{1}}{N}-\frac{u_{2}}{N}}\right)+4 \pi \sum_{j=1}^{n} \delta_{p_{j}}(x)
	\end{aligned}
\end{equation} on $G$, where $n$, $N$ are positive integers, $m_e$, $m_g$ are constants and $\delta_{p_{j}}$ is the dirac mass at vetex $p_j$.

 Let $\mu: V \to (0,+\infty)$ be a finite measure, and $|V|$=$ \text{Vol}(V)=\sum \limits_{x \in V} \mu(x)$ be the volume of $V$.

 We state our main result as follows.
\begin{theorem}\label{t1}
Equations \eqref{1} admits a unique solution if and only if 	\begin{equation}\label{}
	|V|>\frac{4 \pi n}{N m_{e}^{2}}+\frac{4 \pi n(N-1)}{N m_{g}^{2}}.
\end{equation}
\end{theorem}

The paper is organized as follows. In Section 2, we introduce preliminaries. Section 3 is devoted to the proof of Theorem \ref{t1}.

\section{Preliminary results}

For each edge $xy \in E$, we suppose that its weight $w_{xy}>0$ and that $w_{xy}=w_{yx}$. For any function $u: V \to \mathbb{R}$, the Laplacian of $u$ is defined by 
\begin{equation}\label{d1}
	\Delta u(x)=\frac{1}{\mu(x)} \sum_{y \sim x} w_{y x}(u(y)-u(x)),
\end{equation}
where $y \sim x$ means $xy \in E$. The gradient form of $u$ is defined by
\begin{equation}\label{g}
	\Gamma(u, v)(x)=\frac{1}{2 \mu(x)} \sum_{y \sim x} w_{x y}(u(y)-u(x))(v(y)-v(x)).
\end{equation}
Denote the length of the gradient of $u$ by
\begin{equation*}
	|\nabla u|(x)=\sqrt{\Gamma(u,u)(x)}=\left(\frac{1}{2 \mu(x)} \sum_{y \sim x} w_{x y}(u(y)-u(x))^{2}\right)^{1 / 2}.
\end{equation*}
We denote, for any function $
u: V \rightarrow \mathbb{R}
$, an integral of $u$ on $V$ by $\int \limits_{V} u d \mu=\sum\limits_{x \in V} \mu(x) u(x)$. For $p \ge 1$, denote $|| u ||_{p}:=(\int \limits_{V} |u|^{p} d \mu)^{\frac{1}{p}}$. As in \cite{ALY}, we define a sobolev space and a norm by 
\begin{equation*}
	W^{1,2}(V)=\left\{u: V \rightarrow \mathbb{R}: \int \limits_{V} \left(|\nabla u|^{2}+u^{2}\right) d \mu<+\infty\right\},
\end{equation*}
and \begin{equation*}
	\|u\|_{H^{1}(V)}=	\|u\|_{W^{1,2}(V)}=\left(\int \limits_{V}\left(|\nabla u|^{2}+u^{2}\right) d \mu\right)^{1 / 2}.
\end{equation*}

The following Sobolev embedding and Poincaré inequality will be used later in the paper.
\begin{lemma}\label{21}
	{\rm (\cite[Lemma 5]{ALY})} Let $G=(V,E)$ be a finite graph. The sobolev space $W^{1,2}(V)$ is precompact. Namely, if ${u_j}$ is bounded in $W^{1,2}(V)$, then there exists some $u \in W^{1,2}(V)$ such that up to a subsequence, $u_j \to u$ in $W^{1,2}(V)$.
\end{lemma}

	\begin{lemma}\label{22}
		{\rm (\cite[Lemma 6]{ALY})}	Let $G = (V, E)$ be a finite graph. For all functions $u : V \to \mathbb{R}$ with $\int \limits_{V} u d\mu = 0$, there 
		exists some constant $C$ depending only on $G$ such that $\int \limits_{V} u^2 d\mu \le C \int \limits_{V} |\nabla u|^2 d\mu$.
	\end{lemma} 

\section{The proof of Theorem \ref{t1}}
 Since $\int\limits_{V} -\frac{4 \pi n}{|V|} + 4 \pi \sum\limits_{j=1}^{n} \delta_{p_{j}}(x) d\mu = 0,$
 the equation
\begin{equation}\label{2}
	\Delta u_{0}=-\frac{4 \pi n}{|V|}+4 \pi \sum_{j=1}^{n} \delta_{p_{j}}(x), \quad x \in V ; \quad u_{0} \leq 0
\end{equation}
admits a solution $u_0$. Let $v_1 =u_1-u_0$, $v_2 =u_2 -u_0$. Then we know $(v_1,v_2)$ satisfies 
\begin{equation}\label{3}
	\begin{aligned}
		&\Delta v_{1}=-N m_{e}^{2}+\frac{4 \pi n}{|V|}+m_{e}^{2}\left(\mathrm{e}^{u_{0}+\frac{v_{1}}{N}+\frac{(N-1)}{N} v_{2}}+[N-1] \mathrm{e}^{\frac{v_{1}}{N}-\frac{v_{2}}{N}}\right), \\
		&\Delta v_{2}=\frac{4 \pi n}{|V|}+m_{g}^{2}\left(\mathrm{e}^{u_{0}+\frac{v_{1}}{N}+\frac{(N-1)}{N} v_{2}}-\mathrm{e}^{\frac{v_{1}}{N}-\frac{v_{2}}{N}} \right).
	\end{aligned}
\end{equation}
Define the energy functional 
\begin{equation}\label{4}
	\begin{aligned}
		J \left(v_{1}, v_{2}\right)=& \int\limits_{V}\left\{\frac{1}{2 m_{e}^{2}}\Gamma(v_1,v_1)+\frac{(N-1)}{2 m_{g}^{2}}\Gamma(v_{2},v_{2})+N \mathrm{e}^{u_{0}+\frac{v_{1}}{N}+\frac{(N-1)}{N} v_{2}}\right.\\
		&\left.+N(N-1) \mathrm{e}^{\frac{v_{1}}{N}-\frac{v_{2}}{N}}-\left(N-\frac{4 \pi n}{m_{e}^{2}|V|}\right) v_{1}+\frac{4 \pi n(N-1)}{m_{g}^{2}|V|} v_{2}\right\} \mathrm{d}\mu.
	\end{aligned}
\end{equation}

We give a necessary condition for the existence of solutions to \eqref{1} by the following lemma.
 \begin{lemma}\label{x}
If \eqref{1} admits a solution, then 
\begin{equation}\label{11}
	N|V|>\frac{4 \pi n}{m_{e}^{2}}+\frac{4 \pi n(N-1)}{m_{g}^{2}}.
\end{equation}
\end{lemma}
\begin{proof}
Integering \eqref{3}, we deduce that
\begin{equation}
	\begin{aligned}
		& \int\limits_{V}\left(\mathrm{e}^{u_{0}+\frac{v_{1}}{N}+\frac{(N-1)}{N} v_{2}}+[N-1] \mathrm{e}^{\frac{v_{1}}{N}-\frac{v_{2}}{N}}\right) \mathrm{d} \mu=N|V|-\frac{4 \pi n}{m_{e}^{2}}, \\
		& \int\limits_{V}\left(\mathrm{e}^{u_{0}+\frac{v_{1}}{N}+\frac{(N-1)}{N} v_{2}}-\mathrm{e}^{\frac{v_{1}}{N}-\frac{v_{2}}{N}}\right) \mathrm{d} \mu=-\frac{4 \pi n}{m_{g}^{2}},
	\end{aligned}
\end{equation}
which is equivalent to 
\begin{equation}\label{89}
	\begin{gathered}
		N \int_{V} \mathrm{e}^{u_{0}+\frac{v_{1}}{N}+\frac{(N-1)}{N} v_{2}} \mathrm{~d} \mu=N|V|-\frac{4 \pi n}{m_{e}^{2}}-\frac{4 \pi n(N-1)}{m_{g}^{2}}, \\
		N \int_{V} \mathrm{e}^{\frac{v_{1}}{N}-\frac{v_{2}}{N}} \mathrm{~d} \mu=\left(N|V|-\frac{4 \pi n}{m_{e}^{2}}\right)+\frac{4 \pi n}{m_{g}^{2}}.
	\end{gathered}
\end{equation}
Then the desired conclusion follows.

We now complete the proof.
\end{proof}

Next, we give a priori bounds for a solution to \eqref{1}.

\begin{lemma}\label{b}
Suppose that $(v,w)$ is a solution of \eqref{1}. Then we have $v<0$, $w<0$ and $v-w< \frac{N}{N-1}$.
\end{lemma}
\begin{proof}
Let $M:=\max\limits_{V} w =w(x_0)$. We claim that $M<0$. Otherwise, $w(x_0)\ge 0$. Thus, we have 
\begin{equation}
	\Delta w (x_0)=m_{g}^{2}\left(\mathrm{e}^{\frac{v}{N}+\frac{(N-1)}{N} w}-\mathrm{e}^{\frac{v}{N}-\frac{w}{N}}\right)+4 \pi \sum_{j=1}^{n} \delta_{p_{j}}(x) \bigg{|}_{x=x_0}>0.
\end{equation}
On the other hand, by \eqref{d1}, we obtain 
\begin{equation}
	\Delta w (x_0)\le 0.
\end{equation}
This is impossible. Thus, we have 
\begin{equation}
	w(x)< 0
\end{equation}
for all $x\in V$. 

Next, we show that $M_1 :=\max\limits_{x \in V} v=v(x_1)<0$. Suppose by way of contradiction that $M_1 \ge 0$. Let \begin{equation*}
	F(t):=e^{\frac{N-1}{N} t}+(N-1)e^{-\frac{t}{N}}.
\end{equation*} Then it is easy to check that 
\begin{equation*}
	F^{'}(t):= \frac{N-1}{N}e^{\frac{-t}{N} }(e^{t}  -1).
\end{equation*}
Thus we have 
\begin{equation*}
	F(t)>F(0)=N,~t<0.
\end{equation*}
It follows that 
\begin{equation*}
	e^{\frac{N-1}{N} t}+(N-1)e^{-\frac{t}{N}} >N,~t<0.
\end{equation*}
Thus, we have 
\begin{equation*}
	\Delta v(x_1)=-Nm_{e}^{2}+m_{e}^{2}(e^{\frac{v}{N}+\frac{N-1}{N}w} +(N-1)e^{\frac{v-w}{N}})+4\pi\sum\limits_{j=1}^{n}\delta_{p_{j}}(x)>0.
\end{equation*}
By \eqref{d1}, we see that $0\ge \Delta v(x_1)$, this a contradiction.
Thus we obtain $v<0$ for all $x\in V$. 

Now, we show that $M_3:=\max\limits_{x \in V} (v-w)=(v-w) (y_0) < N ln \frac{N}{N-1} $. Assume the assertion is false, then we deduce that 
\begin{equation}
	\begin{aligned}
		\Delta\left(\frac{v}{N}-\frac{w}{N}\right) (y_0)&=\left(\frac{m_{e}^{2}}{N}-\frac{m_{g}^{2}}{N}\right) \mathrm{e}^{\frac{v}{N}+\frac{N-1}{N} w}+\left(\frac{N-1}{N} m_{e}^{2}+\frac{m_{g}^{2}}{N}\right) \mathrm{e}^{\frac{v}{N}-\frac{w}{N}}-m_{e}^{2}\bigg|_{y=y_0}	\\
		&> \frac{N-1}{N} m_{e}^{2} \mathrm{e}^{\frac{v-w}{N}}-m_{e}^{2} \bigg{|}_{y=y_0} \\
		&\ge 0 .
	\end{aligned}
\end{equation}
By \eqref{d1}, we have 
\begin{equation}
	0\ge \Delta\left(\frac{v}{N}-\frac{w}{N}\right) (y_0).
\end{equation}
This is impossible. Thus we have 
\begin{equation}
	v-w < N \ln \frac{N}{N-1} \le \frac{N}{N-1} 
\end{equation}
for all $x\in V$.
\end{proof}

Let $\lambda_1 = m_{e}^{2}$, $\lambda_2 = m_{g}^{2}$, $v=v_1$ and $w=v_2$ in \eqref{3}. Then we have 
\begin{equation}\label{51}
\Delta v=\lambda_{1}\left(\mathrm{e}^{u_{0}} \mathrm{e}^{\frac{v}{N}+\frac{N-1}{N} w}+(N-1) \mathrm{e}^{\frac{v-w}{N}}-N\right)+\frac{4 \pi n}{|V|}, 
\end{equation}	
\begin{equation}\label{52}
\Delta w=\lambda_{2}\left(\mathrm{e}^{u_{0}} \mathrm{e}^{\frac{v}{N}+\frac{N-1}{N} w}-\mathrm{e}^{\frac{v-w}{N}}\right)+\frac{4 \pi n}{|V|}.
\end{equation}
 In order to prove Lemma \ref{y}, we need the following lemma.
 \begin{lemma}\label{EP}
 	Suppose that $u$ satisfies $\Delta u= f$ and $\int_{V} u d\mu =0$. Then we there exists $\hat{C}>0$ such that $$\max\limits_{x \in V} |u(x)| \le \hat{C}||f||_{L^{2}(V)}.$$
 \end{lemma}
\begin{proof}
	From $\Delta u= f$, we deduce that 
	\begin{equation}
		\int_{V} \Gamma(u,u)d \mu=-\int\limits_{x \in V} fu d\mu.
	\end{equation}
By Cauchy inequality with $\epsilon(\epsilon>0)$ and Lemma \ref{22}, there exists $C>0$ such that 
\begin{equation}\label{w1}
		\int_{V} \Gamma(u,u)d \mu \le \frac{1}{4\epsilon} \int_{V} f^{2} d\mu + \epsilon C	\int_{V} \Gamma(u,u)d \mu. 
\end{equation}
Taking $\epsilon=\frac{1}{2C}$ in \eqref{w1}, we have 
\begin{equation}
	\int_{V} \Gamma(u,u)d \mu \le  C \int_{V} f^{2} d\mu .
\end{equation}
Applying Lemma \ref{22}, we know that 
\begin{equation}
	||u||_{L^{2}(V)} \le C||f||_{L^{2}(V)}.
\end{equation}
 Then we deduce that there exists constant $\bar{C}>0$ such that 
\begin{equation}
|u(x)| \le \bar{C} ||f||_{L^{2}(V)}	
\end{equation}
for all $x\in V$.

We now complete the proof.
\end{proof}

To show that Theorem \ref{t1}, we need the following Lemma.
\begin{lemma}\label{y}
	Let $\lambda_{1}=m_{e}^{2}$ and $\lambda_{2}=m_{g}^{2}$. Set $\{ (v_k,w_k)\}$ be a sequence of solutions to equations \eqref{51}-\eqref{52} with $\lambda_{1}=\lambda_{1,k}$ and $\lambda_{2}=\lambda_{2,k}$. Assume that $\lambda_{1,k}\to \lambda_{1}$, $\lambda_{2,k}\to \lambda_{2}$ and \begin{equation}
		\sup \left\{\left|v_{k}(x)\right|+\left|w_{k}(x)\right| \mid x \in V \right\} \rightarrow \infty
	\end{equation} 
as $k\to +\infty$. Then $\lambda_{1}$ and $\lambda_{2}$ satisfy 
\begin{equation}
	|V|=\frac{4 \pi n}{N \lambda_{1}}+\frac{4 \pi n (N-1)}{N \lambda_{2}}.
\end{equation}
\end{lemma}
\begin{proof}
	Denote 
	\begin{equation}\label{54}
			\Delta v_{k}= \lambda_{1, k}\left(\mathrm{e}^{u_{0}} \mathrm{e}^{\frac{v_{k}(x)}{N}+\frac{N-1}{N} w_{k}(x)}+(N-1) \mathrm{e}^{\frac{v_{k}-w_{k}}{N}}-N\right)+\frac{4 \pi n}{|V|}:=f_{k},		
	\end{equation}
\begin{equation}\label{55}
	\Delta w_{k}= \lambda_{2, k}\left(\mathrm{e}^{u_{0}} \mathrm{e}^{\frac{v_{k}}{N}+\frac{N-1}{N} w_{k}}-\mathrm{e}^{\frac{v_{k}-w_{k}}{N}}\right)+\frac{4 \pi n}{|V|}:=g_{k}.
\end{equation}
Denote $\bar{v}_{k}:=\int\limits_{V} v_{k} d \mu$ and $\bar{w}_{k}:= \int\limits_{V} w_{k} d \mu$. Since $\int\limits_{V} v_{k}-\bar{v}_{k}=0$, by Lemma \ref{EP} and Lemma \ref{b}, we deduce that there exists $C_{N}>0$ so that 
\begin{equation}\label{54,}
	\max\limits_{V}(|v_{k}-\bar{v}_{k}|)\le C_{1} ||f_{k}||_{L^{2}(V)} \le C_{N}
\end{equation}
and
\begin{equation}\label{55,}
\max(|w_{k}-\bar{w}_{k}|)\le C_{2} ||g_{k}||_{L^{2}(V)} \le C_{N}.	
\end{equation}
Suppose $\sup\limits_{V} \left\{\left|v_{k}(x)\right| \mid x \in V \right\} \rightarrow \infty$. Since $v_{k}+u_{0}<0$, we deduce that 
$$\bar{v}_{k}\le - \int\limits_{V} u_{0} d \mu.$$ From \eqref{54,}, we deduce that $v_{k}(x)\to -\infty$ and $\bar{v}_{k}\to -\infty$ uniformly on $V$ as $k\to +\infty$. From Lemma \ref{b}, we see that $$\bar{v}_{k}-\bar{w}_{k} \le \frac{N}{N-1} |V|.$$ 

Suppose that 
$$
\liminf _{k \rightarrow \infty}\left(\bar{v}_{k}-\bar{w}_{k} \right)=-\infty.
$$

Subject to passing a subsequence, we have 
$$
\lim _{k \rightarrow \infty}\left(\bar{v}_{k}- \bar{w}_{k}\right)=-\infty.
$$
From \eqref{54,} and \eqref{55,}, we deduce that $$v_{k}(x)-w_{k}(x)\to -\infty~\text{uniformly~on~}V\text{as}~k\to +\infty.$$
It follows that $f_{k} \to -N \lambda_{1} + \frac{4\pi n}{|V|}$. It follows from \eqref{54,} that, by passing to a subsequence, $v_{k}-\bar{v}_{k} \to v(\text{say})$. Letting $k\to +\infty$ in $\Delta (v_{k}-\bar{v}_{k})=f_{k}$. Then we have $\Delta v= -N \lambda_{1} + \frac{4\pi n}{|V|}$ on $V$. This implies that $$N\lambda_{1}|V|=4\pi n.$$ 
By Lemma \ref{x}, we deduce that 
\begin{equation}\label{010}
	N|V|>\frac{4 \pi n}{ \lambda_{1, k}}+\frac{4 \pi (N-1) n}{ \lambda_{2, k}},
\end{equation}
and hence that $|V|>\frac{4\pi n}{\lambda_{1} N}.$ This is impossible. Thus $\{\bar{v}_{k}-\bar{w}_{k} \}$ is bounded. Therefore, $\bar{w}_{k}\to -\infty$ as $k\to \infty$. By \eqref{55,}, we see that $$w_{k}\to -\infty~ \text{as} ~k\to \infty.$$ By passing to a subsequence, we have 
\begin{equation}
	v_{k}-\bar{v}_{k} \rightarrow v,~  w_{k}-\bar{w}_{k} \rightarrow W \text { and } \bar{v}_{k}-\bar{w}_{k} \rightarrow \sigma .
\end{equation}
uniformly~for~ $x\in V $ as~$k\to \infty$. Thus, we deduce that 
\begin{equation}
	\begin{aligned}
		&\Delta v=\lambda_{1}\left((N-1) \mathrm{e}^{\frac{v-W+\sigma}{N}}-N\right)+\frac{4 \pi n}{|V|}, \\
		&\Delta W=\frac{4 \pi n}{|V|}-\lambda_{2} \mathrm{e}^{\frac{v-W+\sigma}{N}},
	\end{aligned}
\end{equation}
and hence that 
\begin{equation}
	\begin{aligned}
		\int_{V} \mathrm{e}^{\frac{v-W+\sigma}{N}} \mathrm{~d} \mu &=\frac{N|V|}{N-1}-\frac{4 \pi n}{\lambda_{1}(N-1)}, \\
		\int_{V} \mathrm{e}^{\frac{v-W+\sigma}{N}} \mathrm{~d} \mu &=\frac{4 \pi n}{\lambda_{2}} .
	\end{aligned}
\end{equation}
Therefore, we conclude that
\begin{equation}\label{11,}
	|V|=\frac{4 \pi n}{N \lambda_{1}}+\frac{4 \pi (N-1) n}{ N\lambda_{2}}.
\end{equation} 

We now complete the proof.
\end{proof}

We will give the proof of Theorem \ref{t1} by applying Lemma \ref{y} and the following Lemma.
\begin{lemma}\label{35}
	Assume that $\lambda_{1}=\lambda_{2}$. Then equations $\eqref{51}-\eqref{52}$ admits a unique solution if and only if $|V|>\frac{4\pi n}{\lambda_{1}}$. 
\end{lemma}
\begin{proof}
Suppose $(v,w)$ is a solution to equations \eqref{51}-\eqref{52}. Due to $\lambda_{1}=\lambda_{2}>0$,  by mean value Theorem, we deduce that there exists $\xi$ such that 
	\begin{equation}
		\Delta (v-w)= \lambda_{1} e^{\xi}(v-w).
	\end{equation} 
Let $M:=\max\limits_{V} (v-w)=(v-w)(x_{0})$. We claim that $M\le 0$. Otherwise, $M>0$. Then $\Delta (v-w)(x_0)=\lambda_{1} e^{\xi} (v-w) \bigg{|}_{x=x_{0}}>0$. By \eqref{d1}, we see that 
\begin{equation*}
	0\ge \Delta(v-w)(x_0).
\end{equation*} 
This is a contradiction. Thus we have $v\le w$ on $V$. By a similar argument as above, we deduce that $v\ge w$ on $V$. Therefore, we conclude that $v\equiv w$ on $V$. Thus, $v$ satisfies 
\begin{equation}\label{58}
	\Delta v=\lambda_{1}(e^{u_0+v}-1) + \frac{4\pi n}{|V|}.
\end{equation}
It follows from \cite{Hub} that \eqref{58} admits a unique solution if and only if $|V|>\frac{4\pi n}{\lambda_{1}}$.
\end{proof}

\begin{proof}[Proof of Theorem \ref{t1}.]	
Define 
$$\bar{H}^{1}(V):=\{ u\in H^{1}(V)| \bar{u}:=\int\limits_{ V} u d \mu =0 \}$$
and
$X:=\bar{H}^{1}(V) \times \bar{H}^{1}(V).$ Let
\begin{equation}
	\begin{aligned}
		&\int\limits_{V} f(x, v(x)+a, w(x)+b) \mathrm{d} x=0 ,\\
		&\int\limits_{V} g(x, v(x)+a, w(x)+b) \mathrm{d} x=0,
	\end{aligned}
\end{equation}
where 
\begin{equation}
	\begin{aligned}
		&f(x, v, w)=\lambda_{1}\left(\mathrm{e}^{u_{0}(x)} \mathrm{e}^{\frac{v}{N}+\frac{N-1}{N} w}+(N-1) \mathrm{e}^{\frac{v-w}{N}}-N\right)+\frac{4 \pi n}{|V|} ,\\
		&g(x, v, w)=\lambda_{2}\left(\mathrm{e}^{u_{0}(x)} \mathrm{e}^{\frac{v}{N}+\frac{N-1}{N} w}-\mathrm{e}^{\frac{v-w}{N}}\right)+\frac{4 \pi n}{|V|}.
	\end{aligned}
\end{equation}
Denote $A=\int\limits_{V} e^{u_{0}+\frac{v}{N}+\frac{N-1}{N} w} d \mu$, $B=\int\limits_{V} e^{\frac{v-w}{N}} d \mu$ and $C=-\frac{N|V|}{4 \pi n} \lambda_{2}+\frac{\lambda_{2}}{\lambda_{1}}$. Then there exists a unique pair 
$$b=b(v,w)=ln \frac{BC+(N-1)B}{A(C-1)},$$
 $$a=a(v,w)=\frac{1}{N} \ln \frac{B C+(N-1) B}{A(C-1)}+\ln \frac{\lambda_{1} N|V|-4 \pi n}{\left(  \frac{BC+(N-1)B}{A(C-1)} A+(N-1) B \right)  \lambda_{1}  } $$
such that 
$$
\begin{aligned}
	&\int_{\Omega} f(x, v(x)+a, w(x)+b) \mathrm{d} x=0, \\
	&\int_{\Omega} g(x, v(x)+a, w(x)+b) \mathrm{d} x=0.
\end{aligned}
$$
For any $(v,w)\in X$, define $$(Q,W):=T(v,w)\in X,$$ where $(Q,W)\in X$ is the unique solution to the equations
\begin{equation*}
	\begin{aligned}
		&\Delta Q=f(x, v+a, w+b), \\
		&\Delta W=g(x, v+a, w+b).
	\end{aligned}
\end{equation*}
By a similar argument as Lemma \ref{EP}, we know that $T$ is completely continuous. Furthermore, by Lemma \ref{y}, there exists $M>0$ such that 
\begin{equation}
	||Q||_{H^{1}(V)}+||W||_{H^{1}(V)}\le M.
\end{equation}
Thus, we may define the Leray-Schauder degree $d(\lambda_{1}, \lambda_{2})$ for $T$. From Lemma \ref{35}, there exists a sufficiently large $\lambda_{0} >0$ so that $d(\lambda_{0}, \lambda_{0})=1.$ In view of \begin{equation*}
	\left\{\left(\lambda_{1}, \lambda_{2}\right)\bigg{|}| V \mid>\frac{4 \pi n}{N \lambda_{1}}+\frac{4 \pi n(N-1)}{N \lambda_{2}}\right\}
\end{equation*} is path-connected. We see that $d(\lambda_{1},\lambda_{2})=d(\lambda_0,\lambda_0)=1$. Therefore, \eqref{51}-\eqref{52} admits at least one solution.
It is easy to check that $J$ defined by \eqref{4} is convex in $H^{1}(V)$. Thus the solution of \eqref{1} is unique.

We now complete the proof.
\end{proof}


\end{document}